\documentclass[reqno]{amsart}
\begin{document}

\theoremstyle{plain}
\begingroup
\newtheorem*{metathm}{Metatheorem} 
\newtheorem{theorem}{Theorem} 
\newtheorem{corollary}{Corollary}
\newtheorem{lemma}{Lemma}
\endgroup
%
%
%
%
\newcommand{\nwc}{\newcommand}
\nwc{\qref}[1]{(\ref{#1})}
\nwc{\cadlag}{c\`{a}dl\`{a}g}
\nwc{\la}{\label}
\nwc{\nn}{\nonumber}
\nwc{\Z}{\mathbb{Z}}
\nwc{\C}{\mathbb{C}}
\nwc{\T}{\mathbb{T}}
\nwc{\E}{\mathbb{E}}
\nwc{\R}{\mathbb{R}}
\nwc{\N}{\mathbb{N}}
\nwc{\Rn}{\mathbb{R}^n}
\nwc{\PP}{\mathcal{P}}
\nwc{\M}{\mathcal{M}}
\nwc{\ad}{\mathrm{ad}}
\nwc{\Ad}{\mathrm{Ad}}
\nwc{\diag}{\mathrm{diag}}
\nwc{\llangle}{\langle\langle}
\nwc{\rrangle}{\rangle\rangle}
\nwc{\hc}{\mathcal{H}}
\nwc{\eig}{\igma}
\nwc{\Ito}{It\^{o}}

\nwc{\law}{\stackrel{\mathcal{L}}{\rightarrow}}
\nwc{\eqd}{\stackrel{\mathcal{L}}{=}}
\nwc{\vp}{\varphi}
\nwc{\veps}{\varepsilon}
\nwc{\eps}{\veps}
\nwc{\dnto}{\downarrow}
\nwc{\nsup}{^{(n)}}
\nwc{\ksup}{^{(k)}}
\nwc{\jsup}{^{(j)}}
\nwc{\nksup}{^{(n_k)}}
\nwc{\inv}{^{-1}}
\nwc{\argmin}{\mathrm{argmin}}
\nwc{\argmax}{\mathrm{argmax}}
\nwc{\Pn}{\mathbb{P}(n)}
\nwc{\PnR}{\mathbb{P}(n;\mathbb{R})}
\nwc{\PnC}{\mathbb{P}(n;\mathbb{C})}
\nwc{\Hn}{\mathbb{H}(n)}
\nwc{\HnR}{\mathbb{H}(n;\mathbb{R})}
\nwc{\HnC}{\mathbb{H}(n;\mathbb{C})}
\nwc{\An}{\mathbb{A}(n)}
\nwc{\AnR}{\mathbb{A}(n;\mathbb{R})}
\nwc{\AnC}{\mathbb{A}(n;\mathbb{C})}
\nwc{\Mn}{\mathbb{M}(n)}
\nwc{\feasible}{\mathcal{F}}
\nwc{\GLn}{GL(n)}
\nwc{\GLnC}{GL(n;\mathbb{C})}
\nwc{\GLnR}{GL(n;\mathbb{R})}
\nwc{\Poincare}{Poincar\'{e}}
\nwc{\Un}{U(n)}
\nwc{\Sn}{\mathbb{S}^{n}}
\nwc{\Tn}{\mathbb{T}^n}

\nwc{\son}{\mathfrak{so}(n)}
\nwc{\so}{\mathfrak{so}}
\nwc{\SO}{\mathsf{SO}}
\nwc{\SOn}{\mathsf{SO}(n)}
\nwc{\Symm}{\mathsf{Symm}}
\nwc{\SymmN}{\mathsf{Symm}(N)}
\nwc{\Pos}{\mathsf{Pos}}
\nwc{\PosN}{\mathsf{Pos(N)}}
\nwc{\PosM}{\mathsf{Pos}(\mathcal{M})}
\nwc{\PosMg}{\mathsf{Pos}(\mathcal{M},g)}
\nwc{\SymmM}{\mathsf{Symm}(\mathcal{M})}
\nwc{\SymmMg}{\mathsf{Symm}(\mathcal{M},g)}
\nwc{\PosqM}{\mathsf{Pos}_q(\mathcal{M})}
\nwc{\PosqMu}{\mathsf{Pos}_q(\mathcal{M},u^\sharp e)}
\nwc{\PosqMg}{\mathsf{Pos}_q(\mathcal{M},g)}
\nwc{\PosqMgt}{\mathsf{Pos}_q(\mathcal{M},g_t)}
\nwc{\PosqMgtplus}{\mathsf{Pos}_q(\mathcal{M},g_t)_+}
\nwc{\SymmqM}{\mathsf{Symm_q}(\mathcal{M})}
\nwc{\SymmqMg}{\mathsf{Symm_q}(\mathcal{M},g)}
\nwc{\PosMplus}{\mathsf{Pos}(\mathcal{M})_+}
\nwc{\PosMgplus}{\mathsf{Pos}(\mathcal{M},g)_+}
\nwc{\PosqMplus}{\mathsf{Pos}_q(\mathcal{M})_+}
\nwc{\PosqMgplus}{\mathsf{Pos}_q(\mathcal{M},g)_+}
\nwc{\MetM}{\mathsf{Met}(\mathcal{M})}

\renewcommand{\Re}{\mathop{\rm Re}\nolimits}
\renewcommand{\Im}{\mathop{\rm Im}\nolimits}

\nwc{\Levy}{L\'{e}vy}
\nwc{\HC}{Harish-Chandra}
\nwc{\BW}{Bures-Wasserstein}
\nwc{\CH}{Cartan-Hadamard}

\theoremstyle{definition}
\newtheorem{defn}[theorem]{Definition} 
\newtheorem{remark}[theorem]{Remark} 
\newtheorem{hyp}[theorem]{Assumption} 
\newtheorem{counterex}[theorem]{Counterexample} 
\newtheorem{ex}[theorem]{Example} 
\theoremstyle{remark}
\newtheorem{notation}{Notation} 
\renewcommand{\thenotation}{} 
\newtheorem{terminology}{Terminology} 
\renewcommand{\theterminology}{} 
\numberwithin{equation}{section}
\numberwithin{figure}{section}

%
%

\nwc{\dd}{\mathfrak{d}}
\nwc{\gln}{\mathfrak{g}}
\nwc{\unn}{\mathfrak{u}}
\nwc{\lotri}{\mathfrak{l}}
\nwc{\GLM}{GL(m,\C)}
\nwc{\proj}{P}
\nwc{\Tr}{\mathrm{Tr}}
\nwc{\Id}{\mathrm{Id}}
\nwc{\MM}{\mathcal{M}}
\nwc{\TT}{\mathbb{T}}
\nwc{\Rq}{\mathbb{R}^q}
\nwc{\GL}{\mathsf{GL}}
\nwc{\schwartz}{\mathcal{S}}
\nwc{\laplacian}{\triangle}
\nwc{\eigen}{\psi}
\nwc{\one}{\mathbf{1}}
\nwc{\Sz}{Sz\'{e}kelyhidi}
\nwc{\Arnold}{Arnol'{d}}
\nwc{\Muller}{M\"{u}ller}
\nwc{\Sverak}{\v{S}ver\'{a}k}
\nwc{\Berard}{B\'{e}rard}
\nwc{\Scal}{\mathrm{Scal}}
\nwc{\Ric}{\mathrm{Ric}}
\nwc{\nnn}{\mathbf{n}}
\nwc{\bb}{\mathbf{b}}
\nwc{\Holder}{H\"{o}lder}
\nwc{\vol}{\mathrm{vol}}
\nwc{\Vol}{\mathrm{Vol}}
\nwc{\rhoeq}{\rho_{\mathrm{eq}}}
\nwc{\matn}{\mathbb{M}}
\nwc{\grad}{\mathrm{grad}}
\nwc{\Lie}{\mathcal{L}}
\nwc{\Hess}{\mathrm{Hess}}

\title{Motion by mean curvature and Dyson Brownian Motion}

\author{Ching-Peng Huang}
\author{Dominik Inauen}
\author{Govind Menon}

\address{Division of Applied Mathematics, Brown University, 182 George St., Providence, RI 02912.}
\address{Institut f\"{u}r mathematik, Universit\"{a}t Leipzig, D-04109, Leipzig, Germany.}
\email{cphuang@brown.edu}
\email{dominik.inauen@math.uni-leipzig.de}
\email{govind\_menon@brown.edu}
%

%


\begin{abstract}
We construct Dyson Brownian motion for $\beta \in (0,\infty]$ by adapting the extrinsic construction of Brownian motion on Riemannian manifolds to the geometry of group orbits within the space of Hermitian matrices. When $\beta$ is infinite, the eigenvalues evolve by Coulombic repulsion and the group orbits evolve by motion by (minus one half times) mean curvature.
\end{abstract}
\maketitle

\section{Introduction}
Fix $\beta \in (0,\infty]$ and $n$ standard independent Wiener processes $\{B_j(t)\}_{j=1}^n$, $t \in [0,\infty)$. Dyson Brownian motion refers to the unique weak solution to the 
\Ito\/ equation
\begin{equation}\label{e:eigenvalueeq}
d\lambda_j =\sqrt{ \frac{2}{\beta}}dB_j + \sum_{k\neq j}\frac{dt}{\lambda_j-\lambda_k}\,,
\end{equation}
within the Weyl chamber
\begin{equation}
\label{eq:weyl-chamber}
\mathcal{W}_n = \{(\lambda_1,\ldots,\lambda_n) \in \R^n\left| \lambda_1 < \lambda_2 \ldots < \lambda_n\right.\}.
\end{equation}
In this paper, we introduce a new stochastic model for equation~\eqref{e:eigenvalueeq}. To this end, we first review the most relevant past constructions to provide some context for our work.

Let $\mu_{n,\beta}$ denote the probability measure on $\mathcal{W}_n$ with density proportional to the weight $\prod_{1\leq j < k \leq n} (\lambda_k-\lambda_j)^\beta$. There are two fundamentally different classes of random matrices whose eigenvalues have law $\mu_{n,\beta}$. For $\beta=1$,$2$ and $4$ these are the self-dual Gaussian ensembles (GOE, GUE and GSE) of real symmetric, complex Hermitian and quaternionic matrices introduced in 
the 1960s by Dyson, Gaudin and Mehta~\cite{Mehta}. For $\beta \in (0,\infty]$, these are the Gaussian $\beta$ ensembles (G$\beta$E) of real, symmetric tridiagonal matrices, introduced by Dumitriu and Edelman~\cite{Dumitriu}. 

Orthogonal polynomials played an important role in the first studies of these ensembles (see~\cite{Deift,Dumitriu,Mehta}). However, dynamic models, especially equation~\eqref{e:eigenvalueeq}, play an important role in understanding universality~\cite{Yau}. Once a random matrix ensemble has been chosen, natural time dynamics can be obtained by replacing a single matrix drawn from the ensemble with a matrix-valued process whose equilibrium measure is the given ensemble. Dyson obtained equation~\eqref{e:eigenvalueeq} in this way for $\beta=1$, $2$ and $4$, replacing each Gaussian self-dual ensemble with its associated Ornstein-Uhlenbeck process~\cite{Dyson}.

For $\beta \in (0, \infty]$ the extension of Dyson's approach to the G$\beta$E ensemble is a subtle problem. Holcomb and Paquette have constructed diffusions of tridiagonal matrices whose eigenvalues satisfy equation~\eqref{e:eigenvalueeq} using orthogonal polynomials and the Lanczos algorithm~\cite{Holcomb}. On the other hand, Yabuoku has studied the eigenvalue process for the G$\beta$E diffusion obtained by choosing independent Ornstein-Uhlenbeck processes on the diagonal and independent Bessel processes on the off-diagonal (see~\cite[(2.1)]{Yabuoku}). He shows that the eigenvalue process of a G$\beta$E diffusion depends on additional minors and does {\em not\/} satisfy Dyson Brownian motion. The gap between these results arises because the diffusions of tridiagonal matrices introduced by Holcomb and Paquette are described somewhat implicitly in terms of a separation between eigenvalue and eigenvector dynamics~(\cite[Section 3]{Holcomb});  these conditions do not hold for the G$\beta$E diffusion.

In the range $\beta \in [0,2]$ a model for equation~\eqref{e:eigenvalueeq} inspired by free probability has been constructed by Allez, Bouchaud and Guionnet~\cite{AG2,AG1}. They construct a stochastic process $S_t$ of real, symmetric matrices whose eigenvalues satisfy~\eqref{e:eigenvalueeq}. Roughly, the process $S_t$ is a scaling limit that interpolates between free convolution and standard convolution steps. Despite the narrower range of $\beta$, and a different matrix model, this work contains certain observations that reappear in~\cite{Holcomb}.

Thus, the existence of natural time-dependent matrix models whose eigenvalues satisfy~\eqref{e:eigenvalueeq} for arbitrary $\beta\in (0,\infty]$ is not fully settled. This is the question we address. 

The main contribution in this work is a geometric interpretation of equation~\eqref{e:eigenvalueeq}. For each $\beta \in (0,\infty]$, we construct a stochastic process $M_t$ in the space of Hermitian matrices (equation~\eqref{e:projectedBM} below) whose eigenvalue process has the same law as the solutions to \eqref{e:eigenvalueeq}. The main new tool in our approach is Riemannian geometry. Specifically, we use Riemannian submersions of group orbits and a probabilistic interpretation of mean curvature to obtain equation~\eqref{e:eigenvalueeq}. The use of Riemannian submersion allows us to view $\beta$ as a parameter that describes an anisotropic splitting between noise in the tangent and normal directions (not an inverse temperature, as in Dyson's work). A similar role for $\beta$ has been observed by Holcomb and Paquette~\cite[Thm.7]{Holcomb}; our approach provides a systematic geometric explanation for its importance. Second, we show that the Coulombic repulsion in equation~\eqref{e:eigenvalueeq} corresponds to the mean curvature of group orbits. This is not a lucky accident: it is a general principle corresponding to the gradient descent of  Boltzmann entropy for group orbits.

In order to explain the main new ideas in the simplest terms, we focus on the explanation of the model, relying on previous work on well-posedness for Dyson Brownian motion and standard calculations in random matrix theory to minimize technicalities. The result in this paper is part of an effort by the authors to develop previously unnoticed connections between three well-studied problems: the construction of Brownian motion on Riemannian manifolds, Dyson Brownian motion, and the isometric embedding problem for Riemannian manifolds. At present, this interplay provides a new formulation of the embedding problem for Riemannian manifolds~\cite{IM1}, new interacting particle systems akin to Dyson Brownian motion~\cite{M-Yu}, and a systematic derivation of SDE for eigenvalue processes of other classes of random matrices using Riemannian submersion~\cite{huang2022eigenvalue}.

\section{Statement of results}
\subsection{The model}
Let $\Hn$ and $\An$ denote the spaces of Hermitian and anti-Hermitian matrices respectively. We equip these spaces with the Frobenius norm $\|M\|^2 = \Tr(MM^*)$. Consider the smooth group action of the unitary group $\Un$ on the open, dense submanifold $V\subset \Hn$ of Hermitian matrices with simple spectrum given by 
\[ \theta:\Un\times V \to V,\, (Q,M)\mapsto \theta(Q,M) = \theta_M(Q) = QMQ^{*}\,.\]
Since $\Un$ is compact the isospectral orbits $\Sigma_M := \theta(\Un,M)$ are embedded smooth submanifolds of $V$ (see Lemma~\ref{lemma:submanifold} below). Consequently, at any point $M\in V$ the tangent space $T_M V = \Hn $ splits into the $n^{2}-n$ dimensional tangent space $T_M\Sigma_M$ to the isospectral orbit through $M$ and the $n$-dimensional normal space  $N_M\Sigma_M$. 
Denote by $P_M:\Hn \to T_M\Sigma_M$ and $P^{\perp}_M:\Hn \to N_M\Sigma_M$ the respective orthogonal projections. 

We construct a process $\{M_t\}_{t\geq 0}$ by a suitable projection of standard Brownian motion on $\Hn$ onto the tangent and normal spaces to isospectral orbits. More precisely, assume given $M_0 \in V$, let $\{X_t\}_{t\geq 0}$ be a standard Brownian motion on $\Hn$ starting at $X_0\in V$, and consider  the \Ito\/ SDE 
\begin{equation}\label{e:projectedBM}
 dM_t = P_{M_t}dX_t + \sqrt{\frac{2}{\beta}} P^{\perp}_{M_t} dX_t\,.
\end{equation}
Using an explicit description of the projection operators and standard SDE theory one can show (Lemma \ref{l:existence}) that for every $\beta >0$ there exists a stopping time $\tau_\beta$ and a solution $M_t$ of \eqref{e:projectedBM} on $[0,\tau_\beta)$.  We then have
\begin{theorem}\label{t:projectedBM}
The eigenvalues of $M_t$ solve the equation \eqref{e:eigenvalueeq} for $t\in [0,\tau_\beta)$.
\end{theorem}
The projection operators are smooth when the spectrum is simple. Since the eigenvalues do not collide when $\beta\geq 1$, a simple bootstrap argument shows that when $\beta \geq 1$, the stopping time $\tau_\beta=\infty$ (see Lemma~\ref{l:existence}).

Theorem~\eqref{t:projectedBM} may be established in a direct way once one has identified the \Ito\/ equation~\eqref{e:projectedBM}. The main insight underlying equation~\eqref{e:projectedBM}, and thus Theorem~\ref{t:projectedBM}, is a probabilistic interpretation of mean curvature. Let us now explain this idea.

\subsection{Brownian motion on Riemannian manifolds and mean curvature}
There are two standard constructions of Brownian motion on Riemannian manifolds using SDEs, referred to as the intrinsic and extrinsic constructions respectively~\cite{Hsu,Ikeda}. The extrinsic construction goes as follows. Assume $\mathcal{M}$ is a smooth $d$-dimensional manifold and assume given a smooth embedding $u: \mathcal{M} \to \mathbb{R}^q$. Then Brownian motion on the embedded submanifold $\Sigma=u(\mathcal{M})$ may be constructed as the solution to the Stratonovich equation
\begin{equation}
    \label{eq:strat-BM}
    dZ_t = P_{Z_t} \circ dW_t 
\end{equation}
where $P_{Z}$ is the orthonormal projection onto $T_Z\Sigma$ in $\mathbb{R}^q$ and $W_t$ is a standard Wiener process in $\mathbb{R}^q$~\cite{Hsu}. 

The use of the Stratonovich formulation is crucial when one studies stochastic processes on manifolds, since it accounts naturally for invariance under coordinate transformations. But equation~\eqref{eq:strat-BM} also admits the equivalent \Ito\/ formulation
\begin{equation}
    \label{eq:strat-BM2}
    dZ_t = P_{Z_t} dW_t + \frac{1}{2}H(Z_t) \, dt
\end{equation}
where $H(Z)$ is the mean curvature vector of the embedding $u$ at the point $Z$. This identity is due to Stroock~\cite[Thm 4.4.2]{Stroock1}; it was rediscoved by two of the authors in their work on the isometric embedding problem~\cite[Thm.2]{IM1}. 

The mean curvature vector of an embedding is defined\footnote{In differential geometry , one frequently finds an additional normalization $H=\frac{1}{n}\mathrm{tr}{II}$. We do not use this normalization factor, adopting the convention of most texts in geometric analysis.} as the trace of the second fundamental form, but equations~\eqref{eq:strat-BM} and \eqref{eq:strat-BM2} show that it may be approached directly from SDE theory. We obtain equation~\eqref{eq:strat-BM2} by beginning with~\eqref{eq:strat-BM}, using the 
conversion rule between the \Ito\/ and Stratonovich formulations to compute the \Ito\/ correction, recognizing finally that the \Ito\/ correction has a fundamental geometric meaning. A related identity involving mean curvature in the case of a  Riemannian submersion was obtained by Pauwels~\cite{Pauwels}.

The intuitive content of equation~\eqref{eq:strat-BM2} is that stochastic fluctuations in the tangent space give rise to a `centrifugal force' given by the mean curvature. Let us illustrate this idea with an example. Let $\Sigma$ be a sphere of radius $r$ in $\R^q$. We compute the projections explicitly, to see that equations~\eqref{eq:strat-BM} and~\eqref{eq:strat-BM2} take the form 
\begin{equation}
    \label{eq:stroock1}
    dZ_t = \left(I  - \frac{Z_t Z_t^T}{|Z_t|^2} \right)\circ dW_t, \quad   dZ_t = \left(I  - \frac{Z_t Z_t^T}{|Z_t|^2} \right) \, dW_t -\frac{q-1}{2r}\frac{Z_t}{|Z_t|}\,dt. 
\end{equation}
The Stratonovich form ensures that the constraint $|Z_t|=r$ holds for all $t$. 
The `centrifugal force' is $(q-1)/2r$ and it arises as follows. If we had naively attempted to construct Brownian motion on the sphere with the \Ito\/ SDE 
\begin{equation}
    \label{eq:stroock3}
    dZ_t = \left(I  - \frac{Z_t Z_t^T}{|Z_t|^2} \right) \, dW_t, \quad\text{we would find}\quad  d|Z_t| = \frac{q-1}{2r|Z_t|}\, dt. 
\end{equation}
Thus, the radial process $|Z_t|$ solves a {\em deterministic\/} equation, even though the evolution of $Z_t$ is purely stochastic. Further, while each point $Z_t$ moves tangentially to the sphere of radius $|Z_t|$, this evolution has the effect of pushing spheres outwards normally by minus a half times the mean curvature. (Observe that the mean curvature vector for the sphere points inward). 

\subsection{Motion by mean curvature}
This observation acquires greater depth when we recall the concept of motion by mean curvature.  A family of immersions $u: \mathcal{M} \times [0,T] \to \mathbb{R}^q$ is said to evolve by motion by mean curvature on the time interval $[0,T]$ if the  velocity of each point $u(x,t)$ is $H(u(x,t))$ for $t \in [0,T]$. Motion by mean curvature has been extensively studied in geometric analysis~\cite{Colding,Giga}. It is related to Dyson Brownian motion as follows. 

\begin{theorem}\label{t:mcf}
Assume $\beta=\infty$ in equation~\eqref{e:projectedBM}. The eigenvalues of the process solving $dM_t = P_{M_t} dX_t$ evolve deterministically by Coulombic repulsion for $t \in [0,\infty)$. Moreover, the corresponding isospectral orbits $\Sigma_{M_t}$ move by minus a half times the mean curvature. 
\end{theorem}
A minor difference with immersions in $\R^q$ is that the mean curvature vector is the trace of the second fundamental form with respect to the Frobenius metric on $\Hn$. The flow is described precisely in equation~\eqref{e:MCF} below.

Theorem~\ref{t:mcf} corresponds to a gradient descent of Boltzmann entropy in the following sense. As in the example above, we see that each matrix $M_t$ on the group orbit moves tangentially (and stochastically), whereas the group orbit as a whole evolves normally (and deterministically) by minus a half times the mean curvature. The group orbits foliate the space $\Hn$ and the group action is an isometry. In this setting, it is known that the mean curvature at each point on the group orbit $\Sigma_{M_t}$ is the gradient (with respect to the Frobenius norm) of $-\log \mathrm{vol}(\Sigma_{M_t}))$~\cite[p.3350]{Pacini}. The volume of the group orbit depends only on the eigenvalues $\Lambda_t$ of $M_t$. By interpreting $\Lambda_t$ as a macrostate, and each point $M_t$ as a microstate, we see that $\log\mathrm{vol}(\Sigma_{\Lambda_t}))$ may be interpreted as a Boltzmann entropy obtained using the theory of Brownian motion on Riemannian manifolds. Further analysis from this viewpoint may be found in~\cite{huang2022eigenvalue,IM1,M-Yu}.

\section{Proofs}
\label{sec:proofs}
We provide self-contained proofs of Theorem~\ref{t:projectedBM} and Theorem~\ref{t:mcf}. An alternative approach, which uses Pauwel's theorem on the relationship between Riemannian submersion and Brownian motion on Riemannian manifolds, and applies to other eigenvalue processes, has been pursued by the first author~\cite{huang2022eigenvalue}. 



\begin{lemma} 
\label{lemma:submanifold}
Assume $M \in V$. The group orbit $\Sigma_M$ is an embedded submanifold in $\Hn$ of dimension $n^{2}-n$. 
\end{lemma}

\begin{proof}
Fix a diagonal matrix $\Lambda\in V$ with the same spectrum, so that $\Sigma_M = \Sigma_\Lambda$. As can be checked, the isotropy group $G_\Lambda = \{Q\in \Un: Q\Lambda Q^{*} = \Lambda\}$ is given by diagonal matrices with entries $\mu_i \in \mathbb S^{1}$ and is hence isomorphic to $\Tn$. Since $G_\Lambda$ is a closed subgroup of $\Un$, the left coset space $\Un/G_\Lambda$ is a smooth manifold of dimension $n^{2}-n$. The orbit map $\theta_\Lambda$ then descends to the quotient and gives a smooth embedding $\theta_\Lambda: \Un/G_\Lambda\to V$. Hence $\dim \Sigma_M = \dim \Un/G_\Lambda   = n^{2}-n$. 
\end{proof}

In the proof of Theorems \ref{t:projectedBM} and \ref{t:mcf} it is convenient to have an explicit description of the tangent and normal spaces to an orbit. 

\begin{lemma}\label{l:tangent} Fix a diagonal matrix $\Lambda\in V$ and $M= Q\Lambda Q^{* }\in \Sigma_\Lambda$. It then holds 
\begin{equation}\label{e:tangent} T_M\Sigma_\Lambda = \{ [QAQ^{*},M]:A\in \An \text{ with } A_{jj}=0 \text{ for } j=1,\ldots,n\}\,,\end{equation}
and 
\begin{equation}\label{e:normal} N_M\Sigma_{\Lambda } = \{Q D Q^{*}: D=\diag(\mu_1,\ldots,\mu_n) \text{ with } \mu_j\in \R\text{ for } j= 1,\ldots,n\}\,.\end{equation}
\end{lemma}

\begin{proof} Fix any skew-hermitian $A\in \An $ and consider the curve 
\[\gamma(t) = \exp(tA)\Lambda \exp(-tA)\,.\] 
Since $\gamma(t)\in \Sigma_\Lambda$ for any $t\in \R$ and $\gamma(0)=\Lambda$, we find that
\[ [A,\Lambda]=\gamma'(0) \in T_\Lambda \Sigma_\Lambda\,.\] 
Since $[iD,\Lambda]=0$ for any real diagonal matrix $D$ it follows by counting dimensions that 
\[ T_\Lambda \Sigma_\Lambda = \{[A,\Lambda]: A\in \An \text{ with } A_{jj}=0 \text{ for } j=1,\ldots,n\}.\,.\]
Observing that $[A,\Lambda]$ has empty diagonal for any $A\in\An$ we find that 
\[ N_\Lambda\Sigma_\Lambda = \{D=\diag(\mu_1,\ldots,\mu_n): \mu_j\in \R\text{ for all } j= 1,\ldots,n\}\,.\]
Since the action $\theta$ restricts to a transitive action on the orbits and $\theta_Q:\Sigma_\Lambda\to \Sigma_\Lambda $ is a local diffeomorphism for any $Q\in \Un$ the claim follows.
\end{proof}

\begin{lemma}\label{l:existence}
For every  $\beta>0$ there exists a stopping time $\tau_\beta$ and a solution $M_t$ of \eqref{e:projectedBM} on $[0,\tau_\beta)$. Moreover, for $\beta\geq 1 $ it holds $\tau_\beta = +\infty$ almost surely. 
\end{lemma}
\begin{proof}
We claim that the maps $M\mapsto P_M$ and $M\mapsto P_{M}^{\perp}$ are smooth on the open, dense subset $V\subset \mathbb{H}(n)$ of hermitian matrices with simple spectrum. The local existence then follows from standard SDE theory with $\tau_\beta$ being (bounded from below by) the first collision time of the eigenvalues. The assertion that $\tau_\beta = +\infty$ almost surely for $\beta\geq 1$ then follows from Theorem \ref{t:projectedBM} and known non-collision results of Dyson Brownian Motion for $\beta\geq 1$ (which are consequences of "McKean's argument", see e.g. \cite[Proposition 4.3]{Mayerhofer}).

To show that the projections are  smooth we fix $M = Q\Lambda Q^{*}\in V$ and $N\in \mathbb{H}(n)$. From the description in Lemma \ref{l:tangent} we find that 
\[ N = P_M(N)+P^{\perp}_M(N) = [QAQ^{*},M]+QDQ^{*}\] 
for some $A\in \mathbb{A}(n)$ with empty diagonal and $D$ diagonal matrix with real entries. Since $Q^{*}[QAQ^{*},M]Q = [A,\Lambda]$ has empty diagonal it follows that 
\[ D= \sum_{j=1}^{n} (Q^{*}N Q)_{jj}e_j\otimes e_j \,,\]
where $e_j$ is the $j$-th standard basis vector of $\R^{n}$. Therefore we have 
\[ P_M^{\perp }(N) = \sum_{j=1}^{n} (Q^{*}N Q)_{jj}Qe_j\otimes e_j Q^{*}=\sum_{j=1}^{n} \mathrm{Tr}(Q^{*}N Q e_j\otimes e_j)P_j=\sum_{j=1}^{n}\mathrm{Tr}(N P_j)P_j\,,\]
where $P_j =P_j(M) := Qe_j(Qe_j)^{*}$ is the (spectral) projection onto the eigenspace associated to eigenvalue $\lambda_j(M)$. Observe that the ordered spectrum $\lambda_1,\ldots,\lambda_n :V\to \R$ is a family of smooth functions on $V$ thanks to the implicit function theorem. Consequently, also the projection operators $P_j$ are smooth on V, which can be seen, for example, from the formula \[ P_j(M) = -\frac{1}{2\pi i } \oint_{\gamma} (M-\lambda \mathrm{Id})^{-1}\,d\lambda\,,\]
where $\gamma \subset\C $ is a Jordan curve such that $\lambda_j(M)$ is the only eigenvalue of $M$ contained in its interior (see e.g. \cite{ReedSimon}). 
\end{proof}

\begin{proof}[Proof (of Theorem \ref{t:projectedBM}):] 
We get the equation for the eigenvalues by It\^o's formula. Consider the ordered spectrum $\lambda_1,\ldots,\lambda_n : V\to \R$. These are smooth functions by the implicit function theorem and the Hadamard variation formulae show that for $A,B\in \Hn $
\begin{equation} \label{e:firstder}
D\lambda_j\vert_{Q\Lambda Q^{*}}(A) = \left (Q^{*} A Q\right )_{jj}\,,
\end{equation} 
and 
\begin{equation} \label{e:secondder}
D^{2}\lambda_j\vert_{Q\Lambda Q^{*}}(A,B) = 2\sum_{k\neq j } \frac{\left (Q^{*} A Q\right )_{kj}\left (Q^{*} B Q\right )_{kj}}{\lambda_j-\lambda_k}\,.
\end{equation} 

Let then $\{E_\alpha\}_{\alpha=1}^{n^{2}}$ be the standard basis (orthonormal with respect to the Frobenius metric) of $\Hn$, i.e.,
\[\{E_\alpha\} = \{e_j\otimes e_j, \frac{1}{\sqrt{2}}(e_k\otimes e_l + e_l\otimes e_k ), \frac{i}{\sqrt{2}} ( e_k\otimes e_l - e_l\otimes e_k)\}\,,\]
where $e_j$ is the $j$-th standard basis vector of $\R^{n}$ and $j=1,\ldots, n$, $1\leq k<l\leq n$. Then \eqref{e:projectedBM} reads 
\[ dM_t = \sum_{\alpha=1}^{n^2}\left (P_{M_t}(E_\alpha) + \sqrt{\frac{2}{\beta}} P^{\perp}_{M_t}(E_\alpha)\right ) dX_t^{\alpha}\,,\] 
where now $X_t^{\alpha}$ are jointly independent standard Wiener processes on $\R$ starting at $X_0^{\alpha}$. By It\^o's formula it follows 
\begin{align*} d\lambda_j &= \sum_{\alpha=1}^{n^2} D\lambda_j\vert_{M_t}\left (P_{M_t}(E_\alpha) + \sqrt{\frac{2}{\beta}} P^{\perp}_{M_t}(E_\alpha)\right ) dX_t^{\alpha} \\
&\quad +\frac{1}{2}\sum_{\alpha=1}^{n^{2}}D^{2}\lambda_j\vert_{M_t}\left (P_{M_t}(E_\alpha) + \sqrt{\frac{2}{\beta}} P^{\perp}_{M_t}(E_\alpha), P_{M_t}(E_\alpha) + \sqrt{\frac{2}{\beta}} P^{\perp}_{M_t}(E_\alpha)\right ) dt\,.\end{align*}

Since $P_{M_t} (E_\alpha)$ is tangent to the isospectral manifold we have $D\lambda_j\vert_{M_t}\left (P_{M_t}(E_\alpha)\right ) = 0$ for any $\alpha$. Moreover, writing $M_t= Q_t\Lambda_t Q^{*}_t$ for a fixed $t>0$ we infer from the description \eqref{e:normal} that $ Q^{*}_t P^{\perp}_{M_t}(E_\alpha)Q_t$ is diagonal and hence 
\[ D^{2}\lambda_j\vert_{M_t}\left (P_{M_t}(E_\alpha),  P^{\perp}_{M_t}(E_\alpha)\right ) = D^{2}\lambda_j\vert_{M_t}\left (P^{\perp}_{M_t}(E_\alpha),  P^{\perp}_{M_t}(E_\alpha)\right ) = 0\] 
for any $\alpha$ thanks to \eqref{e:secondder}. Consequently, 
\[d\lambda_j = \sqrt{\frac{2}{\beta}}\sum_{\alpha=1}^{n^2}D\lambda_j\vert_{M_t}\left(P^{\perp}_{M_t}(E_\alpha)\right ) dX_t^{\alpha}  + \frac{1}{2}\sum_{\alpha=1}^{n^{2}}D^{2}\lambda_j\vert_{M_t}\left (P_{M_t}(E_\alpha), P_{M_t}(E_\alpha)\right ) dt\,.\]
Observe now that 
\[Z_t :=\sum_{\alpha=1}^{n^2}\int_0^{t}D\lambda_j\vert_{M_s}\left(P^{\perp}_{M_s}(E_\alpha)\right ) dX_s^{\alpha} \] is a real-valued martingale with quadratic variation 
 \begin{align*}\langle Z \rangle _t &= \sum_{\alpha=1}^{n^{2}}\int_0^{t} \left (D\lambda_j\vert_{M_s}\left(P^{\perp}_{M_s}(E_\alpha)\right )\right )^{2}ds\ =\sum_{\alpha=1}^{n^{2}} \int_0^{t}\left (D\lambda_j\vert_{M_s}\left(E_\alpha\right )\right )^{2}ds \\
& = \sum_{\alpha=1}^{n^{2}} \int_0^{t}\left ( \left ( Q^{*}_sE_\alpha Q_s\right )_{jj}\right )^{2} ds =   \sum_{\alpha=1}^{n^{2}}\int_0^{t}\langle  E_\alpha, Q_se_j\otimes e_j Q^{* }_s\rangle ^{2}ds \\
&= \int_0^{t}|Q_se_j\otimes e_j Q^{* }_s |^{2}\,ds = t \,,
\end{align*}
since $\{E_\alpha\}$ is an orthonormal basis. By L\'evy's characterisation it therefore follows that $dZ = dB_j$ for a real-valued Brownian motion $B_j$, i.e. ,
\[ d\lambda_j = \sqrt{\frac{2}{\beta}}dB_j + \frac{1}{2}\sum_{\alpha=1}^{n^{2}}D^{2}\lambda_j\vert_{M_t}\left (P_{M_t}(E_\alpha), P_{M_t}(E_\alpha)\right ) dt
\,.\]
Now
\begin{align*} \frac{1}{2}\sum_{\alpha=1}^{n^{2}}D^{2}\lambda_j\vert_{M_t}\left (P_{M_t}(E_\alpha), P_{M_t}(E_\alpha)\right ) &=   \frac{1}{2}\sum_{\alpha=1}^{n^{2}}D^{2}\lambda_j\vert_{M_t}\left (E_\alpha,E_\alpha\right ) \\
&=  \sum_{\alpha=1}^{n^{2}}\sum_{k\neq j } \frac{\left (\left (Q^{*}_t E_\alpha Q_t\right)_{kj}\right )^{2}}{\lambda_j-\lambda_k} \\
&= \sum_{\alpha=1}^{n^{2}}\sum_{k\neq j} \frac{\langle E_\alpha, Q_te_j\otimes e_k Q^{*}_t\rangle ^{2} }{\lambda_j-\lambda_k} \\
&= \sum_{k\neq j} \frac{|Q_te_j\otimes e_k Q^{*}_t|^{2}}{\lambda_j-\lambda_k}=\sum_{k\neq j} \frac{1}{\lambda_j-\lambda_k}\,,
\end{align*}
which shows the claim. 
\end{proof}

\begin{proof}[Proof (of Theorem \ref{t:mcf}):]
Suppose that $dM_t= P_{M_t}dX_t$.  As in the proof of Theorem \ref{t:projectedBM} it follows that 
\begin{align} d\lambda_j &= \sum_{\alpha=1}^{n^2}D\lambda_j\vert_{M_t}\left (P_{M_t}(E_\alpha)\right ) dX_t^{\alpha}  +\frac{1}{2}\sum_{\alpha=1}^{n^{2}}D^{2}\lambda_j\vert_{M_t}\left (P_{M_t}(E_\alpha),P_{M_t}(E_\alpha) \right ) dt\nonumber\\
& = \sum_{k\neq j }  \frac{dt}{\lambda_j-\lambda_k} \label{e:coulombic}\,,\end{align}
which shows the deterministic Coulombic repulsion of the eigenvalues. 

We moreover claim that the orbits $\Sigma_{M_t}$ evolve by time-reversed, scaled mean curvature flow. More precisely, let $\Lambda_t$ be a diagonal matrix with the same spectrum as $M_t$, so that $\Sigma_{M_t} = \Sigma_{\Lambda_t}$. As in Lemma \ref{l:tangent} we consider the closed subgroup $T\subset \Un$ of diagonal matrices with entries in $\mathbb{S}^{1}$. We then define the family of embeddings 
\[ F: \Un/T \times [0,+\infty[ \to \Hn, F(QT,t) = Q\Lambda_tQ^{*}\,,\]
so that $\Sigma_{M_t} = F(\Un/T,t)$. We claim that 
\begin{equation}\label{e:MCF}
\frac{\partial}{\partial t}F(QT,t) = -\frac{1}{2} H_{Q\Lambda_tQ^{*}}\,,
\end{equation}
where $H_M$ is the mean curvature vector of $\Sigma_M$ at $M$, i.e., 
\[ H_M = \mathrm{trace}\, \mathrm{ II}_M\] 
for $\mathrm{II}_M: T_M\Sigma_M\times T_M\Sigma_M \to N_M\Sigma_M$ the second fundamental form of the embedding $\Sigma_M\subset \Hn$. 

First, observe that $H_{Q\Lambda Q^{*}} = QH_\Lambda Q^{*}$, so that it suffices to show \eqref{e:MCF} for $Q=Id$. Indeed, if $\{X_\alpha\} $ is an orthonormal basis of $T_\Lambda \Sigma_\Lambda$, then $Y_\alpha = QX_\alpha Q^{*} = d\theta_Q\vert_\Lambda(X_\alpha)$ is an orthonormal basis of $T_{Q\Lambda Q^{*} }\Sigma_\Lambda$, since $d\theta_Q\vert_\Lambda$ is an isometry. Thus, setting $M=Q\Lambda Q^{*}$, 
\[ H_M = \sum_{\alpha=1}^{n^{2}-n}\mathrm{II}_M(Y_\alpha,Y_\alpha)\,.\]
Since $\Hn$ is flat,  $\mathrm{II}_M(Y_\alpha,Y_\alpha)$ is given by 
\[\mathrm{II}_M(Y_\alpha,Y_\alpha) = P^{\perp}_M\left (\frac{d^{2}}{dt^{2}}\vert_{t=0}\gamma(t)\right )\,,\]
for $\gamma:I\to \Sigma_M$ with $\gamma(0)=M$ and $\gamma'(0)=Y_\alpha$. We can take $\gamma(t)= Q\tilde \gamma(t)Q^{*}$, where $\tilde \gamma :I \to \Sigma_M$  satisfies $\tilde \gamma(0)= \Lambda$ and $\tilde \gamma'(0)=X_\alpha$, so that 
\[ \mathrm{II}_M(Y_\alpha,Y_\alpha) = P^{\perp}_M\left ( Q \frac{d^{2}}{dt^{2}}\vert_{t=0} \tilde \gamma(t) Q^{*} \right ) \,.\] However, from the description \eqref{e:normal} it follows 
\[ P^{\perp}_{Q\Lambda  Q^{*}}(Q H Q^{*} ) = Q P^{\perp}_\Lambda (H) Q^{*}\] 
for any $H\in \Hn$,  from which we infer $H_{Q\Lambda Q^{*}} = QH_\Lambda Q^{*}$.

Assume therefore $Q=Id$ and fix $X=[A,\Lambda] \in T_\Lambda\Sigma_\Lambda$ (recall the description from Lemma \ref{l:tangent}). In the following, for a given $H\in\Hn$ we denote by $D(H)$ the diagonal matrix with the same diagonal as $H$.  Defining $\gamma(t) = \exp(tA)\Lambda\exp(-tA) $ for $t\in \R$ we find 
\begin{equation}\label{e:secondff} \mathrm{II}_{\Lambda} (X,X) = P^{\perp}_\Lambda\left (\frac{d^{2}}{dt^{2}}\vert_{t=0 }\gamma(t)\right ) =D \left ([A,[A,\Lambda]]\right ) = D\left ( [A,X]\right )\,.\end{equation}
To compute $H_\Lambda$ we pick the explicit orthonormal basis of $T_\Lambda\Sigma_\Lambda$ given by 
\[ e_{kl}:= \frac{1}{\sqrt{2}} \left ( e_k\otimes e_l + e_l\otimes e_k\right )\,, iE_{kl} := \frac{i}{\sqrt{2}}\left (e_k\otimes e_l - e_l\otimes e_k\right )\,,\] for $1\leq k <l \leq n$. Observe that  
\[ e_{kl} = \left[ \frac{1}{\lambda_l-\lambda_k}E_{kl},\Lambda\right], iE_{kl}=\left[\frac{i}{\lambda_l-\lambda_k} e_{kl},\Lambda\right] \in T_\Lambda\Sigma_\Lambda\,,\]
so that the collection is indeed an orthonormal basis of $T_\Lambda\Sigma_{\Lambda}$. It follows by \eqref{e:secondff} that 
\[ \mathrm{II}_{\Lambda}(e_{kl},e_{kl}) = \frac{1}{\lambda_l-\lambda_k} D([E_{kl},e_{kl}]) = \mathrm{II}_{\Lambda}(iE_{kl},iE_{kl})\,.\]
But $[E_{kl},e_{kl}] = e_k\otimes e_k - e_l\otimes e_l$, so 
\begin{align*} 
 H_\Lambda &= 2\sum_{k<l}\mathrm{II}_\Lambda(e_{kl},e_{kl}) =  2\sum_{k<l} \frac{1}{\lambda_l - \lambda_k} \left (e_k\otimes e_k-e_l\otimes e_l\right ) \\
 &  = 2 \sum_{k=1}^{n} \left (\sum_{l\neq k} \frac{1}{\lambda_l-\lambda_k} \right )e_k\otimes e_k\,.
\end{align*}
Consequently, we find from \eqref{e:coulombic}
\[ \frac{\partial}{\partial t} F( T,t ) = \sum _{j=1}^{n} \lambda_j' e_j\otimes e_j = \sum_{j=1}^{n} \left (\sum_{k\neq j}\frac{1}{\lambda_j-\lambda_k}\right )e_j\otimes e_j = -\frac{1}{2}H_{\Lambda_t}\,,\]
as claimed.
\end{proof}




\end{document}